\documentclass[12pt,oneside]{amsart}
 
\usepackage{geometry} 
\usepackage{amsmath, amsthm, amssymb} 

\usepackage{graphicx} 
\usepackage{mathrsfs} 
\usepackage{color}
\usepackage[dvipsnames]{xcolor}
\usepackage{graphicx}

\makeatletter
\@addtoreset{equation}{section}
\makeatother

\usepackage[english]{babel}
\usepackage[all]{xy}

\makeatletter
\def\@settitle{\begin{center}%
\baselineskip14\p@\relax
\bfseries
\@title
\end{center}%
}
\makeatother 

\title{Standard models of degree $1$ del Pezzo fibrations}
\author{Konstantin Loginov}
\thanks{Partially supported  by the Russian Academic Excellence Project '5-100', the Simons Foundation and the grants RFBR-15-01-02164 and MK-6019.2016.1.}
\date{} 

\newcounter{cthm}

\newtheorem{thmb}[cthm]{Theorem}
\newtheorem{proposition}[equation]{Proposition}
\newtheorem{thm}[equation]{Theorem} 
\newtheorem{cor}[cthm]{Corollary}
\newtheorem{corr}[equation]{Corollary}
\newtheorem{lem}[equation]{Lemma}
\theoremstyle{definition}
\newtheorem{defin}[equation]{Definition}
\newtheorem{remark}[equation]{Remark}
\theoremstyle{remark}

\newcommand{\simC}{\mathrel{\underset{\scriptscriptstyle{C}}{\sim}}}
\newcommand{\simQ}{\mathrel{\underset{\scriptscriptstyle{\mathbb{Q}}}{\sim}}}
\newcommand{\simQC}{\mathrel{\underset{\scriptscriptstyle{\mathbb{Q}, C}}{\sim}}}
\newcommand{\OOO}{\mathscr{O}}
\newcommand{\HHH}{\mathscr{H}}
\newcommand{\LLL}{\mathscr{L}}

\begin{document}
\newcommand{\Addresses}{{
  \bigskip
  \footnotesize
  \textsc{Laboratory of Algebraic Geometry, Faculty of Mathematics \\ National Research University Higher School of Economics \\
    119048 Moscow, Usacheva str., 6}\\
  \textit{E-mail:} \texttt{kostyaloginov@gmail.com}

}}
\maketitle

\begin{abstract}
We construct a standard birational model (a model that has Gorenstein canonical singularities) for the three-dimensional del Pezzo fibrations ${ \pi \colon X \longrightarrow C }$ of degree $1$ and relative Picard number $1$. We also embed the standard model into the relative weighted projective space $\mathbb{P}_C (1,1,2,3)$. Our construction works in the $G$-equivariant category where~$G$ is a finite group.
\end{abstract}

\section*{Introduction}
The Minimal Model Program (the MMP for short, see \cite{Matsuki2002}, \cite{KMM-1987}) is a powerful tool that helps to understand the birational properties of algebraic varieties. The minimal category in which it works is the category $\mathcal C$ of the projective varieties with at worst terminal $\mathbb Q$-factorial singularities. The result of applying this program to a projective variety is either a minimal model, that is a variety $X\in \mathcal C$ whose canonical divisor class $K_X$ is nef, or a Fano\textendash Mori fibration, that is a variety $X\in \mathcal C$ admitting a contraction morphism $\pi \colon X \longrightarrow B$ whose fibers are of positive dimension, the anti-canonical divisor class $-K_X$ is relatively ample, and $\rho (X/B) = 1$.

We will focus on the three-dimensional case. In this case, the base $B$ of the Fano\textendash Mori fibration $\pi \colon X \longrightarrow B$ can be of dimension $0$, $1$ or $2$. If $\dim B=0$ then $X$ is a (possibly singular) Fano variety. It is known that Fano varieties lie in a finite number of algebraic families. 
However, they are classified only in the smooth case. In the singular case there are partial classificational results (see e.g. \cite{Prokhorov-2016-2}).
 
If $\dim B=2$ then a general fiber of $\pi$ is a non-singular rational curve. In this case, $\pi \colon X \longrightarrow B$ is called a \textit{$\mathbb Q$-conic fibration}. It is known that in this case there exists  a standard model, that is a $\mathbb Q$-conic fibration $\pi' \colon X' \longrightarrow B'$ where $X'$ and $B'$ are non-singular, $X'$ is fiberwise birationally equivalent to $X$, and  $\rho (X'/B') = 1$ (see \cite{Sarkisov}).

Finally, if $\dim B=1$ then a general fiber of $\pi$ is a non-singular del Pezzo surface. In this situation the fibration $\pi \colon X \longrightarrow B$ is called a \textit{$\mathbb{Q}$-del Pezzo fibration}. Standard models of such fibrations were considered in the work of A.~Corti \cite{Corti-1996}, see also \cite{Ko-1997}.

For the applications to the problem of classification of the finite subgroups in the Cremona group (see e. g. \cite{Prokhorov-Shramov}), as well as for birational classification of varieties over algebraically non-closed fields one should change the category $\mathcal C$. We will work with the varieties defined over an arbitrary field of characteristic $0$ that admit an action of a finite group $G$. In this case, we can apply the $G$-equivariant Minimal Model Program (the $G$-MMP, see \cite[2.18, 2.19]{KM-1998}, \cite[0.3.14]{Mori-1988}). Again, a final product of applying this program can be either a $G$-minimal model, or a $G$-Fano\textendash Mori fibration. For a three-dimensional $G$-Fano\textendash Mori fibration $\pi \colon X \longrightarrow B$, as in the ``classical'' situation, we have three possibilities:
\begin{itemize}
\item
$G\mathbb{Q}$-Fano varieties. In general, this class is poorly understood. Partial results can be found in the works \cite{Prokhorov-2015}, \cite{Prokhorov-2016}, \cite{Prokhorov-2016-2}, \cite{Prokhorov-Shramov}. 
\item
$G\mathbb{Q}$-conic fibrations. In this case existence of the standard model is proven in the work \cite{Avilov-conic-r}.
\item
$G\mathbb{Q}$-del Pezzo fibrations (see Definition \ref{defin-1}). In the present work, we construct standard models of $G\mathbb{Q}$-del Pezzo fibrations of degree $1$.
\end{itemize}

The following theorems are the main results of this paper (the necessary definitions are given in \S \ref{section-1}).

\begin{thmb}
\label{thma-A}
Let $X$ be a projective three-dimensional $G$-variety and $C$ be a $G$-curve. Let $\pi \colon X \longrightarrow C$ be a proper $G$-morphism whose generic fiber is a non-singular degree $1$ del Pezzo  surface $X_\eta$, and $\mathrm{Pic}^G (X/C)$ is generated by $-K_X$ and $G$-components of fibers of $\pi$. Then there exists a \emph{Gorenstein model}, that is a generalised $G$-del Pezzo fibration $\sigma \colon Y \longrightarrow C'$ such that 
\begin{enumerate}
\item
the following diagram is commutative 
\[
\xymatrix{
X \ar@{-->}[r]^\chi\ar@{->}[d]^{\pi} & Y\ar[d]_{\sigma}
\\
C\ar@{-->}[r]&C'
} 
\]
where $\chi$ is a birational $G$-equivariant map,
\item
$Y$ has only $\mathbb{Q}$-factorial canonical Gorenstein singularities, 
\item
$C'$ is non-singular and projective, 
\item
$\chi$ induces an isomorphism between $X_\eta$ and $Y_{\eta'}$ where $Y_{\eta'}$ is the generic fiber of $\sigma$, 
\item
any fiber of $\sigma$ is reduced and irreducible.
\end{enumerate}

\end{thmb}

\begin{cor}
\label{thma-B}
If $\pi: X \longrightarrow C$ is a $G\mathbb{Q}$-del Pezzo fibration of degree $1$ then it has a model with at worst $\mathbb{Q}$-factorial canonical Gorenstein singularities, with irreducible fibers and with the same generic fiber as $\pi$.
\end{cor}

\begin{thmb}
\label{thma-C}
Let $\sigma \colon Y \longrightarrow C $ be a generalised $G$-del Pezzo fibration of degree~$1$, and let Y have only Gorenstein canonical singularities. Then $Y$ admits an embedding over~$C$ into the relative weighted projective space $$Y \hookrightarrow \mathbb{P}_C (1,1,2,3).$$
\end{thmb}

We prove these results in several steps. First, in \S \ref{section-1}, Proposition \ref{claim-4}, we give the main definitions and prove some preliminary results. Second, in \S \ref{section-2} we establish some rigidity properties for del Pezzo surfaces and del Pezzo fibrations of degree $1$. Third, in \S \ref{section-3}, Proposition \ref{claim-3}, starting from a del Pezzo fibration of degree~$1$ as in Theorem \ref{thma-A}, we show that $X$ is $G$-birational over $C$ to a $G\mathbb{Q}$-del Pezzo fibration of degree $1$. After that, in \ref{theorem-7} we construct a canonical model of $X$, that is a fibration $\bar{\pi} \colon \bar{X} \longrightarrow C$ which is $G$-birational to $X$ over $C$ and such that the pair $(\bar{X}, | -K_{\bar{X}}+\bar{\pi}^*\bar{D} |)$ is canonical for some $\bar{D}$. Next, in \S \ref{section-4}, Theorem \ref{theorem-10}, we construct a Gorenstein model which proves Theorem~\ref{thma-A}. After that, in $\S \ref{section-5}$, we recall some facts on the anticanonical algebra of degree $1$ del Pezzo surfaces. Finally, in \S \ref{section-6}, we embed a Gorenstein $G$-fibration  
$Y$ into $\mathbb{P}_C (1,1,2,3)$ proving Theorem \ref{thma-C}.

\

The author is grateful to his scientific advisor Yu. Prokhorov for posing the problem and constant support in  writing the paper, and to A. Kuznetsov and C. Shramov for many helpful discussions.   

\section{Preliminaries}
\label{section-1}

We work over a field of characteristic $0$, not necessarily algebraically closed. We also fix a finite group $G$. Recall the standard definitions.

\begin{defin}
\label{defin-G1}
An algebraic variety $X$ is called a \emph{$G$-variety} (or \emph{a variety with an action of the group $G$}) if there exists a homomorphism (not necessarily injective) $${\phi: G \longrightarrow \mathrm{Aut}\ X}.$$ 
\end{defin}

\begin{defin}
\label{defin-G2}
A rational map of $G$-varieties $f: X \dashrightarrow Y$ is called a \emph{$G$-map} if $f$ commutes with the action of the group $G$ on $X$ and on $Y$. If the map $f$ is birational we say that $X$ is \emph{$G$-birational} to $Y$. If the map $f$ is a morphism then $f$ is called a \emph{$G$-morphism}.
\end{defin}

\begin{defin}
\label{defin-G3}
A $G$-variety $X$ is called \emph{$G\mathbb{Q}$-factorial} if any $G$-invariant Weil divisor is $\mathbb{Q}$-Cartier.
\end{defin}
 
Let us fix the notation. For a vector space (or a vector bundle) $A$ we denote its $k$-th symmetric power by $S^k A$, and its full symmetric power by $S^\bullet A$. The base locus of a linear system~$\LLL$ on~$X$ we denote by $\mathrm{Bs} \ \LLL$. Let $\pi: X \longrightarrow C$ be a proper morphism onto the variety $C$. By $\eta$ we denote the generic point of $C$, and by $X_\eta$ the generic fiber of $\pi$. By a general fiber we mean a fiber over some closed point in an open subset $U \subset C$. We denote by $Z_1(X/C)$ a free abelian group generated by reduced irreducible curves which are mapped to points by $\pi$. There is a natural intersection pairing $$ (\ ,\ ) : \mathrm{Pic} (X) \times Z_1 (X/C) \longrightarrow \mathbb{Z}.$$ We put $\mathrm{Pic} (X/C) = \mathrm{Pic} (X) / \equiv$ where $\equiv$ is the numerical equivalence with respect to the pairing introduced above. We denote by $\rho (X/C)$ the dimension of $\mathrm{Pic} (X/C) $ and by $\rho^G (X/C)$ the dimension of the $G$-invariant subspace $\mathrm{Pic}^G (X/C)$.


Let~$D$ and~$D'$ be divisors on $X$. We write $D \simC D'$ if $D \sim D' + \pi^* E $ for some Cartier divisor $E$ on $C$. If $\mathbb{Q}$-divisors on $X$ are $\mathbb{Q}$-linearly equivalent we write $D \simQ D'$. Finally, we write $D \simQC D'$ if $\mathbb{Q}$-divisors $D$ and $D'$ on $X$ are $\mathbb{Q}$-linearly equivalent over $C$, that is $D \simQ D' + \pi^* E$. 

We introduce the main definitions.

\begin{defin}
\label{defin-delpezzo}
\emph{A del Pezzo surface} $S$ is a (not necessarily normal) projective surface that has at worst Gorenstein singularities and whose anti-canonical divisor class~$-K_S$ is ample. \emph{The degree} of a del Pezzo surface $S$ is the number $(-K_S)^2$.
\end{defin}

\begin{defin}
\label{defin-1}
Let $X$ be a three-dimensional normal projective $G$-variety with at worst terminal singularities and let $C$ be a non-singular $G$-curve. Assume that 

\begin{enumerate}
\item\label{defin-1a}
$X$ is $G\mathbb{Q}$-factorial;

\item \label{defin-1b}
there exists a projective $G$-morphism with connected fibers $\pi \colon X \longrightarrow C$;

\item \label{defin-1c}
$-K_X$ is $\pi$-ample;

\item \label{defin-1d}
$\pi$ is an extremal contraction, that is $\rho^{G} (X / C) = 1$.
\end{enumerate}

Then $\pi \colon X \longrightarrow C$ is called a \emph{$G\mathbb{Q}$-del Pezzo fibration}. \emph{The degree of a $G\mathbb{Q}$-del Pezzo fibration} is the degree of its generic fiber $X_\eta$. Since $X$ is terminal  $X_\eta$ is a non-singular del Pezzo surface. If $X$ has Gorenstein singularities we call $\pi \colon X \longrightarrow C$ a \emph{$G$-del Pezzo fibration}.

\end{defin}

\begin{defin}
\label{defin-2}
We call $\pi \colon X \longrightarrow C$ a \emph{generalised $G\mathbb{Q}$-del Pezzo fibration} if $X$ has at worst canonical singularities and the conditions \ref{defin-1b} and \ref{defin-1c} of the definition \ref{defin-1} are satisfied. Notice that in this case the generic fiber $X_\eta$ can be singular.
\end{defin}

We will work with the anticanonical linear system on $X$.

\begin{proposition}[{\cite[2.17, 2.19]{Alexeev-1994ge}}]
\label{claim-4}
Let $\pi \colon X \longrightarrow C$ be a $G\mathbb{Q}$-del Pezzo fibration. Then there exists a divisor $D$ (which we may assume to be $G$-invariant) on $C$ such that the linear system $\HHH=| -K_X+\pi^*D |$ on $X$ is non-empty and has no fixed components.
\end{proposition}

We will use the language of singularities of the linear systems, introduced in \cite[1.8]{Alexeev-1994ge} and \cite{Corti-1995}. It is easy to see that for a $G\mathbb{Q}$-del Pezzo fibration the restriction of the linear system~$\HHH$ chosen above to a general fiber of the morphism $\pi$ is surjective. If the degree of a general fiber is $1$ then the  linear system $\HHH$ has one simple base-point on it. Hence it is easy to see that the pair $(X, \HHH)$ is canonical outside a finite number of fibers. Our aim is to construct a canonical model for the pair $(X, \HHH)$. We will need the following lemmas:

\begin{lem}
\label{lem-14}
Let $\pi \colon X \longrightarrow C$ be a generalised Gorenstein $G$-del Pezzo fibration (that is, $K_X$ is Cartier) of degree $d$ with at worst canonical singularities. Let $F$ be the scheme fiber over a closed point. Write $F=\sum m_i F_i$ where $F_i$ are irreducible components. Then $\sum m_i \le d$. In particular, if $d=1$, then any geometric fiber is reduced and irreducible.
\end{lem}
\begin{proof}
We may assume the ground field to be algebraically closed. By the adjunction formula $ K_X |_F = K_F $, and we have $$d=(-K_F)^2=(-K_X)^2\cdot F = (-K_X)^2 \cdot \sum m_i F_i =\sum m_i (-K_F |_{F_i})^2 \geq \sum m_i,$$ the last inequality holds since $-K_F$ is an ample Cartier divisor on $F_i$. 
\end{proof}

\begin{lem}[{\cite[1.22]{Alexeev-1994ge}}]
\label{lem-5}
 Suppose that the pair $(X, \HHH)$ is terminal where $\HHH$ is a linear system without fixed components. Then $\HHH$ has at worst isolated non-singular base-points $P_i$ such that $\mathrm{mult}_{P_i} \HHH = 1$. 
\end{lem}

\begin{lem}[{\cite[1.23]{Alexeev-1994ge}}]
\label{lem-6}

Suppose that $X$ has only terminal singularities and the pair $(X, \HHH)$ is canonical. Then in the neighbourhood of any base-point $P$ of $\HHH$ we have $\HHH \sim -K_X$.
\end{lem}

\section{Birational rigidity}
\label{section-2}

The following lemma is a variant of birational rigidity of degree $1$ del Pezzo surfaces (cf. \cite[1.6]{Is-1996}).

\begin{lem}
\label{lem-15}
Let $S$ be a non-singular del Pezzo surface of degree $1$, and let $T$ be either a normal del Pezzo surface of degree $d$ or a conic bundle over a non-singular curve. Let $f: S \dashrightarrow T$ be a birational map between them. Put 
$$\HHH =
\begin{cases}
| - K_T | ,\ \text{if } d \geq 3, \\
| - 2 K_T | ,\ \text{if } d =2, \\
| - 3 K_T |,\ \text{if } d =1, \\
| nF |, \text{where $F$ is the class of a fiber if T is a conic bundle, $n\geq 1$.}
\end{cases}$$
Suppose that $\LLL : = f^{-1}_* \HHH \subset | - a K_S | $ for some positive integer $a$. Then $T$ is a degree $1$ del Pezzo surface and $f$ is an isomorphism. 
\end{lem}
\begin{proof}
All the properties in the claim can be checked over an algebraically closed field, so we assume that. Consider a resolution of the points of indeterminacy of $f$
\[
\xymatrix{
& Z \ar[dl]_g \ar[dr]^h &
\\
S \ar@{-->}[rr]^{f} & & T 
} 
\]
Consider the case where $T$ is a del Pezzo surface. Since in this case $\HHH$ is very ample (see Proposition \ref{claim-12}) $h$ is the  blow up of the base locus of $\LLL$. It is easy to check (see \cite[1.3.2]{Is-1996}) that for the strict transform $\widetilde{\LLL} := g^{-1}_* \LLL $ we have
$$ \widetilde{\LLL}^2 = \LLL^2 - \sum r_i ^2$$
$$ K_Z \cdot \widetilde{\LLL} = K_S \cdot \LLL + \sum r_i$$
where $r_i$ are the multiplicities of the base points of $\LLL$ (including infinitely near ones). Notice that $\widetilde{\LLL} = h^{-1}_* \HHH$, and since $\HHH$ is base point free we get 
$$ \HHH^2 = \widetilde{\LLL}^2 = \LLL^2 - \sum r_i ^2 = a^2 - \sum r_i^2 $$
$$ \HHH \cdot K_T = K_Z \cdot \widetilde{\LLL} = - a + \sum r_i $$

Consider the cases:

1) $d\geq 3$. We get
$$a^2 = \sum r_i^2 + d, \ \ \ a = \sum r_i + d.$$
It follows that $r_i= 0$ for any $i$ and $d = 1$ which contradicts the assumption. 

2) $d = 2$. We get
$$a^2 = \sum r_i^2 + 8, \ \ \ a = \sum r_i + 4.$$
These equations easily lead to a contradiction.

3) $d = 1$. Then we get the equations
$$a^2 = \sum r_i^2 + 9, \ \ \ a = \sum r_i + 3.$$
We deduce that all the $r_i = 0$, hence the map $f$ is a morphism and $a = 3$. Thus $f$ is a contraction of the exception divisor $E$. Hence $$K_S = f^* K_T + E$$
On the other hand, 
$aK_S = f^*K_T$. Hence $(1 - a)K_S = E$ which is absurd. We see that there are no contracted curves and $f$ is an isomorphism.

Now consider the case when $T$ is a conic bundle over a non-singular curve. That is there is a morphism $\tau: T \longrightarrow B$ whose general fiber is a non-singular conic. 
We have $\HHH^2 = (nF)^2 = 0$ and $\HHH\cdot K_T = nF\cdot K_T = -2n$ by adjunction. Considering the resolution of the base points of $\LLL$ we can write the formulas as above and get $$ a^2 = \sum r_i^2, \ \ \ a = \sum r_i +2n. $$ Again, it is easy to derive a contradiction.

We see that only the case 3) can occur, and the claim follows. 
\end{proof}

The next proposition gives a generalization of the rigidity property to degree~$1$ del Pezzo fibrations. 

\begin{proposition}
\label{claim-19}
Let $X$ be a projective three-dimensional $G$-variety and $C$ be a non-singular $G$-curve. Let $\pi \colon X \longrightarrow C$ be a $G$-morphism whose generic fiber is a non-singular degree $1$ del Pezzo  surface $X_\eta$, and $\mathrm{Pic}^G (X/C)$ is generated by $-K_X$ and $G$-components of fibers. Suppose that $X$ is $G$-birational over $C$ to a generalised $G\mathbb{Q}$-Fano\textendash Mori fibration $\pi': X' \longrightarrow B$ over $C$, that is the following diagram is commutative
\[
\xymatrix{
X \ar@{-->}[r]^\chi \ar[d]^{\pi} & X'\ar[d]_{\pi'}
\\
C& \ar@{->}[l]_{\psi} B
} 
\]
Then $\psi$ is an isomorphism, $X'$ is a $G\mathbb{Q}$-del Pezzo fibration of degree $1$ and $X_\eta \simeq X'_\eta$. Here $X'_\eta$ is the generic fiber of $X'$ over $C$.
\end{proposition}
\begin{proof}
The map $\chi$ induces a birational map $f: X_\eta \dashrightarrow X'_\eta$. 
Suppose that $B$ is a curve. Since the diagram is commutative, the fibers of $\pi$ and $\pi'$ are connected and $B$ is normal we get that $\psi$ is an isomorphism. Then $X'_\eta$ is a del Pezzo surface over the function field of $C$. We want to apply Lemma \ref{lem-15}. 

First consider the case when $X'_\eta$ is a degree $1$ del Pezzo surface. By adjunction we have $-3K_{X'} |_{X'_\eta} = - 3 K_{X'_\eta}$. The class of $-3K_{X'}$ is $G$-invariant in $\mathrm{Pic} (X'/C)$. Since $\chi$ is a $G$-map, the class of $\chi^{-1}_* ( -3K_{X' })$ is also $G$-invariant in $\mathrm{Pic} (X/C)$. Hence it has the form $-a K_X + D$ for some $a$ and some divisor $D$ concentrated in the fibers. Since $(-aK_{X} + D) |_{X_\eta} = - a K_{X_\eta}$ we have $f^{-1}_* ( 3 K_{X'_\eta} ) = - a K_{X_\eta}$, and the conditions of Lemma \ref{lem-15} are satisfied. Thus $f$ is an isomorphism. 


If $X'_\eta$ is a del Pezzo surface of degree $d\geq 2$ a similar argument yields a contradiction. 

Now suppose that $B$ is a surface in which case $X'_\eta$ as a surface over the function field of $C$ admits a conic bundle structure. Consider the divisor class $F_\eta$ on $X'_\eta$ corresponding to the generic fiber of the map $\pi'|_{X'_\eta}: X'_\eta \longrightarrow B_\eta$. Clearly there is a divisor $F$ on $X'$ such that $F|_{X'_\eta} = F_\eta$. Put $F'=\sum_{g\in G} g. F$. Since $G$ sends a fiber of $\pi'$ to a fiber of $\pi'$ we have $(g.F)|_{X'_\eta} = F_\eta$ for any $g \in G$. Hence $F'$ is $G$-invariant in $\mathrm{Pic}\ X$ and $F'|_{X'_\eta} = n F_\eta$ where $n = |G|$. Since $\chi$ is a $G$-map, $\chi^{-1}_*F'$ is $G$-invariant as well, so it has the form $-aK_X + D$ for some $a$ and some divisor $D$ concentrated in the fibers of $\pi$. Hence $f^{-1}_* (n F_\eta) = -aK_{X_\eta}$, so the conditions of Lemma~\ref{lem-15} are satisfied, and we get a contradiction.
\end{proof}

\section{Canonical model}
\label{section-3}

First we construct a model that is a $G\mathbb{Q}$-del Pezzo fibration in the sense of Definition \ref{defin-1}.

\begin{proposition}
\label{claim-3}
Let $X$ be a projective three-dimensional $G$-variety and $C$ be a $G$-curve. Let $\pi \colon X \longrightarrow C$ be a proper $G$-morphism whose generic fiber is a non-singular degree $1$ del Pezzo surface $X_\eta$, and $\mathrm{Pic}^G (X/C)$ is generated by $-K_X$ and $G$-components of fibers. Then there exists a $G\mathbb{Q}$-del Pezzo fibration $\pi' \colon X' \longrightarrow C'$ with the generic fiber $X'_\eta$ isomorphic to $X_\eta$ such that the following diagram is commutative: 
\[
\xymatrix{
X\ar@{-->}[r]^\chi\ar@{->}[d]^{\pi} & X'\ar[d]^{\pi'}
\\
C \ar@{-->}[r] & C'
} 
\]
\end{proposition}

\begin{proof}
Let $j:~C_1 \longrightarrow C$ be a normalization. Consider the following commutative diagram
\[
\xymatrix{
X \ar@{->}[d]^{\pi} \ar@{-->}[dr]^{\pi_1} &
\\
C \ar@{<-}[r]^{j} & C_1
} 
\]
where the map $\pi_1:= \pi \circ i^{-1}$ may be not defined over some points of $C_1$. 

Let $X_1$ be a $G$-equivariant resolution of singularities (see \cite[Theorem 0.1]{AW-1997}) of $X$ and the indeterminacy points of $\pi_1$:
\[
\xymatrix{
X\ar@{<-}[r]^{h} \ar@{->}[d]^{\pi} & X_1 \ar@{->}[d]^{\pi_2}  
\\
C \ar@{-->}[r] & C_1  
} 
\]

Since a general fiber of $\pi_1$ is non-singular $h$ does not change it. Apply the $G$~-MMP (see \cite[0.3.14]{Mori-1988}) over $C_1$ to the variety $X_1$. We get the following diagram of $G$-maps where the map $g$ is a composition of flips and divisorial contractions
\[
\xymatrix{
X_1 \ar@{-->}[r]^g \ar[d]^{\pi_2} & X'\ar[d]_{\pi'}
\\
C_1& \ar@{->}[l]_{\psi} B
} 
\]

Put $\chi = g \circ h^{-1}$. The map $\chi$ induces a birational map of the generic fibers $f: X_\eta \dashrightarrow X'_\eta$ where $X'_\eta$ is the generic fiber of the morphism $\psi \circ \pi'$. Then Proposition \ref{claim-19} shows that $f$ and $\psi$ are isomorphisms, and the claim follows. 
\end{proof}



Notice that the assumption on $\mathrm{Pic}^G (X/C)$ in the above proposition is a relaxation of property \ref{defin-1a} in the Definition \ref{defin-1}. Now we are ready to construct a canonical model. 

\begin{thm}
\label{theorem-7}
Let $\pi \colon X \longrightarrow C$ be a $G\mathbb{Q}$-del Pezzo fibration of degree $1$. Then there exist a $G\mathbb{Q}$-del Pezzo fibration $\bar{\pi} \colon \bar{X} \longrightarrow C$ of degree $1$ with the generic fiber $\bar{X}_\eta \simeq X_\eta$ and a commutative diagram
\[
\xymatrix{
X \ar@{-->}[rr]^{h}\ar[dr]^{\pi} && \bar{X} \ar[dl]_{\bar{\pi}} 
\\
& C  &
} 
\]
where the map $h$ is birational and the pair $(\bar{X}, \bar{\HHH})$ is canonical where ${\bar{\HHH}=| - K_{\bar{X}}+\bar{\pi}^*\bar{D} |}$ for some ample $G$-invariant divisor $\bar{D}$ on $C$. The pair $(\bar{X}, \bar{\HHH})$ is called a \emph{canonical model} of $(X, \HHH)$.
\end{thm}

\begin{proof}
We put $\HHH=| - K_X+\pi^*D |$ where $D$ is a $G$-invariant ample divisor on $C$ as in Proposition \ref{claim-4}. We may assume that the following map is surjective 
$$ \mathrm{H}^0 (X, \OOO_X (-K_X+\pi^*D)) \twoheadrightarrow \mathrm{H}^0 (
S, \OOO_S (-K_X+\pi^*D))=\mathrm{H}^0 (S, \OOO_S (-K_S)).$$
for a general fiber $S$ of the morphism $\pi$. Thus on $S$ the linear system $\HHH$ has only one simple base-point. There is a corresponding rational section of the morphism $\pi$ whose closure we will denote by $\Gamma$. Thus $\Gamma \subset \operatorname{Bs} \HHH $. 

Let $g_1 \colon X_1 \longrightarrow X$ be the blow-up of the curve $\Gamma$. We denote its exceptional divisor by $E_1$. The linear system $\HHH_1 = g_*^{-1} \HHH$ is base-point free on a general fiber of the morphism $\pi_1$.

Let $g_2  \colon \widetilde{X} \longrightarrow X_1$ be a $G$-equivariant resolution of singularities of the pair $(X_1, \HHH_1)$. We have a commutative diagram
\[
\xymatrix{
\widetilde{X} \ar[r]^{g_2}  \ar[dr]_{\widetilde{\pi}}  & X_1 \ar[d]^{\pi_1}   \ar[r]^{g_1}& X \ar[dl]^{\pi} 
\\
& C &
} 
\]
We introduce the notation
$$ g = g_2 \circ g_1, \ \ \widetilde{\HHH} = (g_2)_*^{-1} \HHH_1, \ \ \widetilde{E}_1 = (g_2)_*^{-1} E_1.$$ 

On a general fiber $g_2$ is an isomorphism. Hence all the exceptional divisors of the morphism $g_2$ are contained in a finite number of fibers of $\widetilde{\pi}$. Since the linear system~$\widetilde{\HHH}$ is base-point free, the pair $(\widetilde{X}, \widetilde{\HHH 
})$ has the same singularities as $\widetilde{X}$ itself. In particular, $(\widetilde{X}, \widetilde{\HHH 
})$ is canonical. Write
$$ 
K_{\widetilde{X}}+\widetilde {\HHH}+\sum a_i \widetilde{E}_i=g^*(K_X+\HHH) \simC 0,\qquad (K_X+\HHH) \simC 0, $$
where $\widetilde{E}_i$ are the exceptional divisors of the morphism $g$. We have $a_i \in \mathbb{Z}$ since $K_X+\HHH$ is a Cartier divisor on $X$. By construction all the exceptional divisors $\widetilde{E}_i$ are contained in a finite number of fibers of the morphism $\widetilde{\pi} \colon \widetilde{X} \longrightarrow C$ except $\widetilde{E}_1$ for which we have $g_* \widetilde{E}_1=\Gamma$.

We run the $G$-MMP over $C$ for the pair ${(
\widetilde{X}, (1 - \epsilon) \widetilde{\HHH})}$ for $0 < \epsilon \ll 1$. 

Let $\widetilde{S}$ be a general fiber of the morphism $\widetilde{\pi}$. It is easy to see that  the linear system $$ - (K_{\widetilde {X}}+(1 - \epsilon) \widetilde{\HHH} )|_{\widetilde{S}} = - \epsilon K_{\widetilde{S}}$$ is nef on $\widetilde{S}$. Hence the result of applying the $G$-MMP is a $G\mathbb{Q}$-Fano\textendash Mori fibration with the base $B$ over $C$. We have the following commutative diagram
\[
\xymatrix{
\widetilde{X} \ar@{-->}[r]^{f}\ar[d]_{\widetilde{\pi}} & \bar{X}\ar[d]^{\bar {\pi}}\ar[dl]^{}
\\
C & B \ar[l]^{\psi}
} 
\]
where $f$ is a composition of divisorial contractions and $K_{\widetilde {X}}+(1 - 
\epsilon) \widetilde{\HHH}$-flips. By Proposition \ref{claim-19} we see that $\bar{X}$ is a $G\mathbb{Q}$-del Pezzo fibration of degree $1$ over $B$ (and $B = C$) and $\bar{X}_\eta \simeq X_\eta$.

The pair $(\widetilde{X}, (1 - \epsilon) \widetilde{\HHH})$ is terminal, therefore the pair $(\bar{X}, (1 - \epsilon) \bar{\HHH})$ is terminal as well. In particular, $\bar{X}$ has terminal singularities. Since $1$ is not an accumulation point of the set of $3$-dimensional canonical thresholds (see \cite{Prokhorov-2008}) we can choose $\epsilon \ll 1$ such that the pair $(\bar{X}, \bar{\HHH})$ is canonical.

Since $f^{-1}$ does not contract divisors the following formula holds:
\begin{equation}
\label{eq:f1} 
K_{\bar{X}}+\bar{\HHH}+\sum a_i \bar{E}_i \simC 0,
\end{equation}
where $f_{*} \widetilde{E}_i = \bar{E}_i$. Thus, 

\begin{equation}
\label{eq:f2} 
\bar{\HHH}+\sum a_i \bar{E}_i \sim - K_{\bar{X}}+\bar{\pi}^*A
\end{equation}
for a $G$-invariant divisor $A$ on $C$. By construction $\bar{E_i}$ lie in the fibers of $\bar{\pi}$. Since~$\rho^G (\bar{X} / C) = 1$ any fiber of $\widetilde{\pi}$ is $G$-irreducible. Since $\bar{E}_i$ are $G$-irreducible we get that $m_i \bar{E}_i=\bar {\pi}^*(x_i)$ where $x_i \in C$ for some integer $m_i$. The case $m_i > 1$ corresponds to a multiple fiber of the morphism~$\bar{\pi}$. Adding if necessary some number of points $x_i$ to $A$ we may assume that in the equation \eqref{eq:f2} we have $0 \leq a_i < m_i$ for any $a_i$. Suppose that there exists such $i$ that~$a_i > 0$. 

By construction we have $f_* \widetilde{E}_1 = \bar{\Gamma}$ where $\bar{\Gamma}$ is a curve that is a section of $\bar{\pi}$ and $\bar{\Gamma} \subset \mathrm{Bs}\ \bar{\HHH}$. By Lemma \ref{lem-5} in a neighbourhood of any base-point of $\bar{\HHH}$ we have $ - K_{\bar{X}} \sim \bar{\HHH}$. If $\bar{\pi}^* ( x_i ) = m_i \bar{E}_i$ then  
$m_i \bar{E}_i \cdot \bar{\Gamma}=1$ and $\bar{E}_i \cdot \bar{\Gamma}=
\frac{1}{m_i}$. Let us consider the base-point $P \in \bar{E}_i \cap \bar{\Gamma}$. From the formula \eqref{eq:f1} it follows that $a_i \bar{E}_i$ is a Cartier divisor in a neighbourhood of $P$. Thus $a_i 
\bar{E}_i \cdot \bar{\Gamma} \in \mathbb{Z}$. On the other hand, \[
0 < a_i \bar{E}_i 
\cdot \bar{\Gamma}=\frac{a_i}{m_i} < 1 
\]
since $0 < a_i < m_i$. The contradiction shows that \[ 
K_{\bar{X}}+\bar{\HHH} \simC 0. 
\] 
This completes the proof.
\end{proof}

\section{Gorenstein model}
\label{section-4}
We need the following technical proposition that follows easily from Kodaira's Lemma (see e. g. \cite[0.3.5]{KMM-1987}).

\begin{proposition}
\label{claim-8}
Let $B$ be a nef big $\mathbb{Q}$-Cartier divisor on a normal projective variety~$X$ with at worst terminal singularities. Then there exists an ample $\mathbb{Q}$-Cartier divisor~$A$ and an effective $\mathbb{Q}$-Cartier divisor $N$ such that $B\simQ A+N$ and the pair~$(X, N)$ is terminal.
\end{proposition}

We also use the construction of a terminal modification of a pair.

\begin{proposition}[{\cite[2.8]{Corti-1995}}]
\label{claim-9}
Let $X$ be a normal projective $G$-variety with at worst terminal singularities and let the pair $(X, \HHH)$ be canonical. Then there exists a $G$-equivariant terminal modification, that is a $G$-variety $\bar{X}$ and a birational $G$-morphism $f \colon \bar{X} \longrightarrow X$ such that the pair $(\bar{X}, \bar{\HHH})$ is terminal where $\bar {\HHH} = f^{-1}_* \HHH$ and the following formula holds
$$
K_{\bar{X}}+\bar {\HHH}=f^*(K_X+\HHH).
$$ 
\end{proposition}

Now we construct a Gorenstein model. 

\begin{thm}
\label{theorem-10}
Let $\pi \colon X \longrightarrow C$ be a $G\mathbb{Q}$-del Pezzo fibration of degree $1$. Then there exists a Gorenstein model, that is a generalised $G\mathbb{Q}$-del Pezzo fibration $\sigma \colon Y 
\longrightarrow C$ of degree $1$ such that $Y$ is $G$-birational to $X$ over $C$ and $Y$ has at worst  canonical Gorenstein singularities. Moreover, the generic fiber $Y_\eta$ of $\sigma$ is non-singular, $Y_\eta \simeq X_\eta$, and the special fibers are reduced and irreducible. 
\end{thm}

\begin{proof}
By Theorem \ref{theorem-7} we may assume that the pair $(X, \HHH)$ is canonical where $\HHH = | -K_X + \pi^*D|$ for some $G$-invariant ample Cartier divisor $D$ on $C$. Let $(\bar{X}, \bar{\HHH})$ be its $G$-equivariant terminal modification over $C$ (see Proposition \ref{claim-9}). By Lemma \ref{lem-5} the linear system $\bar{\HHH}$ has at worst isolated non-singular base-points $P_i$ such that $\mathrm{mult}_{P_i} \bar{\HHH} = 1$. Since $\bar{X}$ is terminal it has at worst isolated singular points. It follows that a general element of $\bar{\HHH}$ is a Cartier divisor.
 
The linear system $\bar{\HHH}$ is nef since it has no curves in its base locus. Restricting $\bar{\HHH}$ to a general fiber of the morphism ${\bar{\pi} \colon \bar{X} \longrightarrow C}$ we notice that for any $m\geq2$ the image of the map given by the linear system $| m \bar{\HHH} |$ is three-dimensional. Hence the linear system $ \bar{\HHH}$ is big. By Proposition \ref{claim-8} for some $\mathbb{Q}$-Cartier divisors $A$ and $N$ we have $$\bar{\HHH}\simQ A+N,$$ where $A$ is ample, $N$ is effective and the pair $(\bar{X}, N)$ is terminal. By the construction of $\bar{X}$ we have $K_{\bar{X}}+\bar{\HHH} \simC 0$, hence $$K_{\bar{X}}+N \simQC - A.$$

Thus, for any curve $Z \in \overline{\operatorname{NE}}(X / C)$ with $Z\cdot \bar{\HHH}=
0$ we have $Z \cdot (K_{\bar{X}}+N) < 0$. By Contraction Theorem \cite[Theorem~3.2.1]{KMM-1987} the linear system $|n \bar{\HHH}|$ gives a ({$G$-equivariant}) contraction morphism $g \colon \bar{X} \longrightarrow Y$ with connected fibers such that the variety $Y$ is projective over $C$ and the following diagram is commutative
\[
\xymatrix{
\bar X \ar[rr]^{g} \ar[dr]^{\bar{\pi}}  & &  Y \ar[dl]_{\sigma}
\\
& C &
} 
\]

It also guarantees the existence of a $\sigma$-ample Cartier divisor $H$ on $Y$ such that $g^*{H}=\bar{\HHH}$. Since $\bar{\HHH}$ is ample on an open subset of $\bar{X}$ we see that $g(\bar{X})=Y$ is three-dimensional.  Hence $g$ is birational. 

Now we prove that $Y$ has Gorenstein canonical singularities. Since by construction of $\bar{X}$ we have $K_{\bar{X}}+\bar{\HHH}=\bar{\pi}^*D$ for some ample divisor $D$ on $C$ and $ g_* (K_{\bar{X}}+\bar{\HHH})=K_Y+H$ then from the commutativity of the diagram $$K_Y+H=g_* \bar{\pi}^*D=\sigma^*D.$$ Thus $K_Y$ is a Cartier divisor. We have $$g^*K_Y=g^*(-H+\sigma^*D)=-\bar{\HHH}+
\bar{\pi}^*D=K_{\bar{X}}.$$ Hence $Y$ has canonical singularities. 

The fact that $Y_\eta$ is non-singular and $Y_\eta \simeq X_\eta$ follows from Proposition \ref{claim-19}. The special fibers are reduced and irreducible due to Lemma \ref{lem-14}. It remains to show that $Y$ is $G\mathbb{Q}$-factorial. We show this in the next lemma.
 
\begin{lem}
Let $\sigma \colon Y \longrightarrow C$ be a generalised Gorenstein del Pezzo fibration of arbitrary degree with at worst canonical singularities. Suppose that the fibers of $\sigma$ are reduced and irreducible. Then $Y$ is $\mathbb{Q}$-factorial.
\end{lem}
\begin{proof}
Since $Y$ is canonical it has finitely many non-$\mathbb{Q}$-factorial points by \cite[3.4]{Reid-1985}. Let $F = \sum F_i$ be the union of the fibers of $\sigma$ that contain all the points that are not $\mathbb{Q}$-factorial. Write the excision exact sequence
$$  \bigoplus_i \mathbb{Z} F_i \longrightarrow \mathrm{Cl} \ Y \longrightarrow \mathrm{Cl} \ U \longrightarrow 0 $$
where $U = Y \setminus F$. Since $U$ is $\mathbb{Q}$-factorial we have $\mathrm{Cl} \ U \otimes \mathbb{Q} = \mathrm{Pic} \ U \otimes \mathbb{Q}$. We see that any Weil divisor on $Y$ after subtracting some number of fibers $F_i$ and taking multiple comes from $\mathrm{Pic} \ U$. So it is enough to prove that any Cartier divisor on $U$ extends to a Cartier divisor on $Y$. But this follows from the commutative diagram with exact rows (see \cite[21.4.3]{EGA-1967})
\[
\xymatrix{
0 \ar[r] & \mathrm{Pic} \ C \ar[r]^{\sigma^*}  \ar[d]  & \mathrm{Pic} \ Y \ar[r] \ar[d]  & \mathrm{Pic} \ Y_\eta \ar[d]^{\simeq} \ar[r] & 0
\\
0 \ar[r] & \mathrm{Pic} \ V \ar[r]^{\sigma|_V^*} & \mathrm{Pic} \ U \ar[r] & \mathrm{Pic} \ Y_\eta  \ar[r] & 0
} 
\]
where $V = \sigma (U)$. Indeed, the left vertical arrow is clearly surjective, hence by the~Snake Lemma the middle vertical arrow is surjective as well. 
\end{proof}

The theorem is proven.
\end{proof}

\begin{remark}
\label{remark-11}
In the proof of the theorem by construction we have $$Y \simeq \operatorname{Proj}_C \bigoplus_{m\geq0} \sigma_* \OOO_Y (m H).$$
\end{remark}

\

\begin{proof}[Proof of Theorem \ref{thma-A}]
Follows immediately from Propositions \ref{claim-4}, \ref{claim-3},  and Theorems \ref{theorem-7}, \ref{theorem-10}.
\end{proof}

\section{Anticanonical algebra of a degree $1$ del Pezzo surface}
\label{section-5}

We need the following results on the anticanonical algebra $$R=\bigoplus_{m\geq0} \mathrm{H}^0 (S, -mK_S)$$ of a degree $1$ del Pezzo surface $S$.

\begin{proposition}
\label{claim-12}
The following holds: 
\begin{enumerate}
\item \label{claim-12a}
$\dim \mathrm{H}^0 (S, -mK_S)=m(m+1)/2+1;$
\item \label{claim-12b}
the linear system $|-K_S|$ has one simple base-point;
\item \label{claim-12c}
the linear system $|-2K_S|$ is generated by global sections;
\item \label{claim-12d}
the linear system $|-3K_S|$ is very ample.
\end{enumerate}
\end{proposition}

\begin{proof}
Let $S$ be normal. In this case, the statements \ref{claim-12a}--\ref{claim-12d} are well known if $S$ is non-singular. In the singular case see \cite[\S 4]{HW-1981}.

Let $S$ be non-normal and let $\alpha \colon T \longrightarrow S$ be its normalisation. According to \cite[1.1]{Reid-1994} (see also \cite[1.5]{AF-2003}) we have $$T \simeq \mathbb{P}^2, \quad \alpha^*(-K_S) \simeq 
\OOO_{\mathbb{P}^2} (1).$$

It is proven there that $\alpha$ is an isomorphism outside a (possibly singular) conic $Q$ on $\mathbb{P}^2$, and the morphism $\alpha|_Q$ is a $2$ to $1$ covering over a curve on $S$.

According to \cite[2.2]{AF-2003} for any $m\geq1$ the following holds:
$$
\dim \mathrm{H}^0 (T, \alpha^*(-mK_S))=\dim \mathrm{H}^0 (S, -mK_S)+
m.$$ 
Hence we get $$\dim \mathrm{H}^0 (S, -mK_S)=\dim 
\mathrm{H}^0 (\mathbb{P}^2, \OOO_{\mathbb{P}^2} (m)) - m=(m+1)(m+2)/2 - m=
m(m+1)/2+1.$$

This proves \ref{claim-12a}. The statements \ref{claim-12b} и \ref{claim-12d} are proven in \cite[1.2, 1.5(A)]{AF-2003}, see also \cite[4.10 (ii)]{Reid-1994}.

Since we have $| \sigma^*(- K_S) |=| 
\OOO_{\mathbb{P}^2}(1)|$, all the elements of the linear system $| -K_S |$ are irreducible. 
By the adjunction formula they are curves of arithmetic genus~$1$. Since they are rational these curves are either nodal or cuspidal cubic curves. For any $C \in |-3K_S|$ we have the following exact sequence 
$$ 0 \longrightarrow \OOO_S (-3K_S-C) \longrightarrow \OOO_S (-3K_S) \longrightarrow \OOO_C (-3K_S) \longrightarrow 0. $$
Since any linear system on the curve $C$ of degree $2$ or more does not have base-points and $\mathrm{H}^1(S, \OOO_S) = 0$, the statement \ref{claim-12c} follows.
\end{proof}

\begin{corr}
\label{cor-13}
\begin{enumerate}
\item\label{corr-13a} 
The linear system $|-2K_S|$ defines a two-fold covering of the quadratic cone $$\phi=\phi_{|-2K_S|}  \colon S \longrightarrow Q \subset \mathbb{P}^3.$$

\item\label{corr-13b} 
The algebra $R=\bigoplus_{m\geq0} \mathrm{H}^0 (S, -mK_S)$ is isomorphic to $ \mathbb{C} [x, y, z, w] / (f) $ where the degree of generators are $1, 1, 2, 3$ correspondingly and the relation $f$ has degree $6$. Hence, $S$ is isomorphic to the degree $6$ hypersurface in the weighted projective space $\mathbb{P} (1, 1, 2, 3)$.
\end{enumerate}
\end{corr}
\begin{proof}
Follows from the previous proposition in full analogy to the non-singular case, see \cite[8.3]{Dolgachev-2012}.
\end{proof}

\section{The relative projective space}
\label{section-6}

We put $$A_m=\sigma_* (\OOO_Y (- m K_Y ) ), \quad m\geq 0, $$ $$ A=\bigoplus_{m=0}^{\infty} 
A_m.$$ 

By Proposition \ref{claim-12} the sheaf $A_m$ on $C$ is a vector bundle of rank $m(m+1)/2+1$. 

\begin{remark}
\label{remark-15}
We can restrict the sheaf of algebras $A$ on a fiber of the morphism $\sigma$ and apply Corollary \ref{cor-13} to show that $A$ is generated by its components of degree~$\leq 3$.
\end{remark}

We will construct the relative weighted projective space $\mathbb{P}_C 
(1,1,2,3)$, that is a variety which is projective over $C$ and which has $\mathbb{P}(1,1,2,3)$ as a fiber over any point of $C$. We will also construct an embedding over $C$ 
\[
\xymatrix{
Y \ar[rr]^{i} \ar[dr]^{\sigma}  & &  \mathbb{P}_C (1,1,2,3) \ar[dl]
\\
& C &
} 
\]

We denote the cokernel of the natural inclusion $\alpha \colon S_2 A_1 \longrightarrow A_2$ by $G_2$. We get an exact sequence
\begin{equation}\label{equation-2}
0 \longrightarrow S^2 A_1 \longrightarrow A_2 \longrightarrow G_2 \longrightarrow 0.
\end{equation}

Consider the multiplication map $\mu \colon A_1 \otimes A_2 \longrightarrow A_3$. Let 
$V = \operatorname{Ker} \mu$, $G_3 = \operatorname{Coker} \mu$. It is easy to check fiberwise that  
$V$ и $G_3$ are vector bundles of rank $2$ and $1$ correspondingly. Hence we get an exact sequence
\begin{equation}\label{equation-3}
0 \longrightarrow (A_1\otimes A_2)/V \longrightarrow A_3 \longrightarrow G_3 \longrightarrow 0 
\end{equation}

Consider a commutative diagram of natural maps
\[
\xymatrix{
S^\bullet (A_1 \oplus A_2 \oplus A_3) \ar[r] & A 
\\
\ar[u] S^\bullet (A_1 \oplus A_2 \oplus ((A_1\otimes A_2)/V)) \ar[ur] & 
} 
\]

By Remark \ref{remark-11} we have $Y=\operatorname{Proj}_C A$. We define the relative projective spaces $$\mathbb{P}_C (1^2, 2^4, 3^7) = \mathrm{Proj}_C \ S^\bullet (A_1 \oplus A_2 \oplus A_3 )$$ $$ \mathbb{P}_C (1^2, 2^4, 3^6) = \mathrm{Proj}_C \ S^\bullet (A_1 \oplus A_2 \oplus ((A_1\otimes A_2)/V) ) ) $$

Thus we obtain a commutative diagram of varieties over $C$
\[
\xymatrix{
\mathbb{P}_C (1^2, 2^4, 3^7) \ar@{-->}[d]^{p} & \ar[l]_-{j} Y \ar@{-->}[dl]^{q}
\\
\mathbb{P}_C (1^2, 2^4, 3^6) &
} 
\]
where $q=p \circ j$. According to Remark \ref{remark-15} the map $j$ is a fiberwise embedding over~$C$. 

\begin{proposition}
\label{claim-16}
The image $ q (Y) $ is a quadratic cone $ \mathbb{P}_C (1,1,2)$ over $C$, that is a projective variety over $C$ such that a fiber over any point $x \in C$ is isomorphic to $ \mathbb{P} (1,1,2)$. Moreover, the following diagram is commutative
\[
\xymatrix{
\mathbb{P}_C (1^2, 2^4, 3^7) \ar@{-->}[d]^{p} & \ar[l]_-{j} \ar[d]_{v} Y \ar[dl]_{q}
\\
\mathbb{P}_C (1^2, 2^4, 3^6) & \ar[l]_-{u} \mathbb{P}_C (1,1,2)
} 
\]
where $u$ is a fiberwise embedding, $v$ is a fiberwise two-fold covering and $p$ is a projection from a point. In particular the map $q$ is a morphism.

\end{proposition}

\begin{proof}
On the fiber over any point $x\in C$ the sequences \eqref{equation-2} and \eqref{equation-3} split as the sequences of vector spaces. We have the following non-canonical isomorphisms \begin{eqnarray}
\label{equation-4}
(A_2)_x &\simeq& (S^2 A_1 \oplus G_2)_x 
\\
\label{equation-5}
(A_3)_x &\simeq& (S^3 A_1 \oplus (A_1 \otimes G_2) \oplus G_3)_x 
\\
\label{equation-6}
( (A_1 \otimes A_2) /V )_x &\simeq& (S^3 A_1 \oplus (A_1 \otimes G_2) )_x 
\end{eqnarray}

On the other hand, over a point $x\in C$ we have $$ \operatorname{Proj} S^\bullet (A_1 \oplus G_2)_x \simeq \mathbb{P} (1,1,2)$$

Using the isomorphisms \eqref{equation-4}--\eqref{equation-6} over a point $x\in C$ we construct a natural surjective map 
$$
\xymatrix@R=10pt{
s \colon S^\bullet 
(A_1 \oplus (S^2 A_1 \oplus G_2) \oplus (S^3 A_1 \oplus (A_1 \otimes G_2)))_x
\ar[r] \ar@{}[d]|*[@]{\simeq} & S^\bullet (A_1 \oplus G_2)_x, 
 \\
S^\bullet (A_1 \oplus A_2 \oplus ((A_1 \otimes A_2) /V)))_x &
}
$$
where $s$ identically maps $$s ( ( S^k A_1 )_x ) = (S^k A_1)_x,\ k\geq 1, \ \   s ( (G_2)_x ) = (G_2)_x$$ and $s$ extends to the corresponding tensor powers in the obvious way. Thus $s$ induces an embedding $u \colon \mathbb{P} (1,1,2) \longrightarrow \mathbb{P} (1^2, 2^4, 3^6)$.

The map $s$ can be extended to a commutative diagram 
\[
\xymatrix{
S^\bullet (A_1 \oplus (S^2 A_1 \oplus G_2) \oplus (S^3 A_1 \oplus (A_1 \otimes G_2) \oplus G_3))_x \ar[r]^-{s'} & A_x 
\\
S^\bullet (A_1 \oplus (S^2 A_1 \oplus G_2)_x \oplus (S^3 A_1 \oplus (A_1 \otimes G_2)))_x \ar[r]^-{s} \ar[u]^{z} \ar[ur] & S^\bullet (A_1 \oplus G_2)_x \ar[u]^{w}
} 
\]
Here the map $z$ is induced by the natural inclusion of vector spaces, and the map $s'$ is constructed in full analogy with the map $s$:
$$ s'(S^k A_1)_x = ( S^k A_1 )_x, \ k\geq1,$$
$$s'(G_2)_x = ( G_2 )_x \subset ( S^2 A_1 \oplus G_2 )_x \simeq (A_2)_x  \subset A_x$$
$$s'(G_3)_x = ( G_3 )_x \subset ( S^3 A_1 \oplus(A_1 \otimes G_2)) \oplus G_3 )_x \simeq ( A_3 )_x \subset A_x.$$

It is easy to check that $w$ is a fiberwise two-fold covering of the quadratic cone as in Corollary \ref{cor-13} and that the diagonal map in the diagram induces the map $q$. This proves the claim.\end{proof}

In the notation as above consider a variety $$Z = \overline{p^{-1} (u (\mathbb{P}_C (1,1,2)))} \subset \mathbb{P}_C (1^2, 2^4, 3^7).$$ Clearly it is projective over $C$.

\begin{proposition}
\label{claim-17}
Each fiber of $Z$ over $C$ is isomorphic to $\mathbb{P} (1,1,2,3)$. Moreover, there exists an embedding $i \colon Y \longrightarrow Z$ over $C$. 
\end{proposition}

\begin{proof}
It is obvious that $Z$ has dimension $3$ over $C$ and that $j(Y) \subset C$. For the fiber $Y_x$ over a point $x\in C$ we construct a commutative diagram
\[
\xymatrix{
\mathbb{P} (1^2, 2^4, 3^7) \ar@{-->}[d]^{p} & \ar[l]_-{r} \ar@{-->}[d]^{t} 
\mathbb{P}(1,1,2,3) \ar@{-->}[dl]^{q} & \ar[l]_-{i} \ar[dl]^-{v} Y_x
\\
\mathbb{P} (1^2, 2^4, 3^6) & \ar[l]^-{u} \mathbb{P} (1,1,2) &
} 
\]
where $u, v, p$ are as in Proposition \ref{claim-16}, $i$ is an embedding, $t$ is a projection from the point $(0:0:0:1)$, and $q=u \circ t$.

Using the isomorphisms \eqref{equation-4}--\eqref{equation-6} we construct a natural surjective map 
$$ 
S^\bullet (A_1 
\oplus A_2 \oplus A_3) \simeq S^\bullet (A_1 \oplus (S^2 A_1 \oplus G_2) \oplus 
(S^3 A_1 \oplus (A_1 \otimes G_2) \oplus G_3)) \longrightarrow S^\bullet (A_1 
\oplus G_2 \oplus G_3) 
$$

This map induces an embedding $r$. We must show that $r(\mathbb{P}(1,1,2,3))$ is contained in $Z_x$ and hence must coincide with it. The variety $Z$ was defined as the closure of the preimage of $u(\mathbb{P}_C (1,1,2))$ under the map $p$. Thus, it is enough to show that there is a fiberwise inclusion $$(p \circ r) (\mathbb{P} (1,1,2,3)) \subset u (\mathbb{P} (1,1,2)).$$ But it follows from the commutativity of the following diagram
\[
\xymatrix{
S^\bullet (A_1 \oplus (S^2 A_1 \oplus G_2) \oplus (S^3 A_1 \oplus (A_1 \otimes G_2) \oplus G_3))_x \ar[r] & S^\bullet (A_1 \oplus G_2 \oplus G_3)_x \ar[r] & A_x
\\
S^\bullet (A_1 \oplus (S^2 A_1 \oplus G_2) \oplus (S^3 A_1 \oplus (A_1 \otimes G_2)))_x \ar[r] \ar[u] \ar[ur] & S^\bullet (A_1 \oplus G_2)_x \ar[u] \ar[ur] &
} \qedhere
\]
\end{proof}

We denote the variety $Z$ by $\mathbb{P}_C (1,1,2,3)$ 

\begin{proof}[Proof of Theorem B]
Follows immediately from Proposition \ref{claim-17}.
\end{proof}


\def\cprime{$'$} \def\mathbb#1{\mathbf#1}

\Addresses

\end{document}